\documentclass[notitlepage,11pt,psamsfonts]{amsart}

\usepackage{amssymb}

\newtheorem{theorem}{Theorem}[section]

\newtheorem{proposition}[theorem]{Proposition}
\theoremstyle{definition}
\newtheorem{definition}[theorem]{Definition}

\theoremstyle{remark}
\newtheorem{remark}[theorem]{Remark}

\numberwithin{equation}{section}

\newcommand{\eps}{\varepsilon}
\newcommand{\R}{\mathbb{R}}

\newcommand{\N}{\mathbb{N}}
\newcommand{\Z}{\mathbb{Z}}

\renewcommand{\div}{{\rm div}\,}

\newcommand{\Id}{{\rm Id}\,}

\newcommand{\Int}{\displaystyle \int}

\def\d{\partial}

\def\dj{\Delta_j}
\def\tilde{\widetilde}

\def\div{{\rm div}\,}

\def\cA{{\mathcal A}}
\def\cB{{\mathcal B}}
\def\cC{{\mathcal C}}
\def\cF{{\mathcal F}}

\def\cP{{\mathcal P}}

\def\cO{{\mathcal O}}

\begin{document}
\title[lifespan for flows of incompressible fluids]
{Remarks on the lifespan of the solutions to some models of incompressible
fluid mechanics}

\author[R. Danchin]{Rapha\"el Danchin}
\address[R. Danchin]
{Universit\'e Paris-Est, LAMA, UMR 8050,
 61 avenue du G\'en\'eral de Gaulle,
94010 Cr\'eteil Cedex, France.}
\email{danchin@univ-paris12.fr}



\begin{abstract}
We give lower bounds for the lifespan of a solution to the inviscid Boussinesq system. 
In dimension two, we point out that it tends to infinity 
when the initial (relative) temperature tends to zero.
This is, to the best of our knowledge, the first result of this kind for the
inviscid Boussinesq system. In passing, we provide continuation criteria
(of independent interest) in the $N$-dimensional case. 
 In the second part of the paper,
our method is adapted to handle  the axisymmetric
incompressible Euler equations with swirl.
  \end{abstract}

\maketitle

\section*{Introduction}

The evolution of the velocity $u=u(t,x)$  and pressure $P=P(t,x)$ fields of a perfect 
homogeneous  incompressible fluid is governed
by the following Euler equations:
\begin{equation}\label{eq:euler}
\left\{\begin{array}{l}
\d_tu+u\cdot\nabla u+\nabla P=0,\\[1ex]
\div u=0.
\end{array}\right.
\end{equation}

There is a huge literature concerning the well-posedness issue for Euler equations.
Roughly, they may be solved locally in time in 
any reasonable Banach space embedded in 
the set $C^{0,1}$ of bounded 
Lipschitz functions (see e.g. \cite{BCD,BM,Ch,hmidiK,HK,PP,V,Z}).

In the two-dimensional case, it is well known that Euler equations are  globally well-posed
for sufficiently smooth initial data. This noticeable fact relies on 
the conservation of the vorticity $\omega:=\d_1u^2-\d_2u^1$ along the flow 
of the velocity field, and has been first proved rigorously in the pioneering 
works by W. Wolibner \cite{W} and V. Yudovich \cite{Yu}.

This  conservation property is no longer true, however, in more 
physically relevant contexts such as
\begin{enumerate}
\item the three-dimensional setting for \eqref{eq:euler},
\item nonhomogeneous  incompressible perfect fluids,
\item inviscid  fluids  subjected to a buoyancy force
which is advected by the velocity fluid (the so-called inviscid Boussinesq system below).
\end{enumerate}

 As a consequence, the problem of global existence for general (even smooth
or small) data is still open for the above three cases. 
\medbreak
In a recent work \cite{DF}, it has been shown that for slightly nonhomogeneous
two-dimensional incompressible fluids, the lifespan tends to infinity
when the nonhomogeneity tends to zero. 
The present paper is mainly dedicated to the study of the lifespan 
for the first and third item. 

More precisely, in the first section of the paper, we shall consider the  
\emph{inviscid Boussinesq system}:
\begin{equation}\label{eq:boussinesq}
\left\{\begin{array}{l}
\d_t\theta+u\cdot\nabla \theta=0,\\[1ex]
\d_tu+u\cdot\nabla u+\nabla P=\theta e_N,\\[1ex]
\div u=0.
\end{array}\right.
\end{equation}
Here the relative temperature $\theta=\theta(t,x)$ is a real valued function\footnote{It need not
be nonnegative as it designates the discrepancy to some 
reference temperature.} and $e_N$ stands for the unit vertical vector.  
\smallbreak

 As for the standard incompressible Euler equations, 
 any functional space   embedded in  $C^{0,1}$ 
is a good candidate  for the study of the well-posedness issue 
for \eqref{eq:boussinesq}. 
This stems from  the fact that System \eqref{eq:boussinesq}  is a  coupling between
  transport equations. 
Hence preserving the initial regularity requires the velocity field to be at least locally Lipschitz with respect to the space variable.
By arguing as in  \cite{BCD}, Chap. 7, 
one may  show that, indeed,
 System \eqref{eq:boussinesq}
is locally well-posed in  $B^s_{p,q}$  whenever
$B^s_{p,q}$ is embedded in $C^{0,1}$ or, in other words, for any $(s,p,q)\in\R\times[1,+\infty]^2$
satisfying \begin{equation}\label{eq:conditionC}
s>1+\frac{N}{p}\qquad\mbox{or}\qquad s=1+\frac{N}{p}\;\;\mbox{and}\;\;q=1\,. 
\end{equation}

As a by-product of estimates for transport equations,  we shall get various continuation
 criteria which generalize those
of \cite{ES} and of \cite{LWZ}. 
We shall finally establish lower bounds for the lifespan of the solutions 
to \eqref{eq:boussinesq} which show that in the two-dimensional case
and for small initial temperature, the solution tends to be global-in-time. 
\medbreak
As pointed out in many works (see e.g. \cite{ES}), 
there is a formal similarity between the two-dimensional Boussinesq system
and  general axisymmetric solutions to the three-dimensional Euler system
-- the so-called axisymmetric solutions \emph{with swirl}. 
In the second part of this paper, we adapt the method of the first part so as to establish
 new lower bounds for the lifespan 
to those solutions in the case where the swirl is small. In particular, we find out  that the solution 
tends to be global if the swirl goes to zero. 
\smallbreak
In the Appendix, we briefly recall the definition and a few basic properties
of Besov space, and prove a commutator estimate. 
\medbreak
  Before going further into the description of our results, let us introduce a few notation.
 \begin{itemize}
\item Throughout the paper, $C$ stands for
a harmless ``constant'' the meaning of which depends on the context.
\item 
The vorticity $\omega$ associated to a vector field $u$ over $\R^N$ is the matrix valued
function with entries
$$
\omega_{ij}:=\d_j u^i-\d_iu^j.
$$
If $N=2$ then the vorticity is identified with the scalar function
$\omega:=\d_1u^2-\d_2u^1$ and if $N=3,$ with the vector field $\nabla\times u.$
\item For all  Banach space $X$ and interval $I$ of $\R,$  
we denote by $\cC(I;X)$ 
 the set of continuous   functions on $I$ with
values in $X.$
If $X$ has predual $X^*$ then 
we denote by $\cC_w(I;X)$ the set of bounded measurable functions 
$f:I\rightarrow X$ such that for any $\phi\in X^*,$ the 
function $t\mapsto\langle f(t),\phi\rangle_{X\times X^*}$
is continuous over~$I.$
\end{itemize}


\section{The inviscid Boussinesq system}

This section is devoted to the well-posedness issue for the inviscid Boussinesq system
\eqref{eq:boussinesq}. We  first establish a local-in-time existence 
result and continuation  criteria in the spirit of those for the incompressible Euler equation. 
Next, we provide a new lower bound for the lifespan. 
Roughly, we  establish that if $\theta_0$ is of order $\eps,$
then the lifespan is at least of order $\log |\log\eps|.$

\subsection{Well-posedness and continuation criteria}

The present subsection is devoted to the proof of the following result.

\begin{theorem} \label{th:boussinesq}
Let $(s,p,q)\in\R\times[1,+\infty]^2$ satisfy \eqref{eq:conditionC}.
Assume that $u_0$ (with $\div u_0\equiv0$) 
and $\theta_0$ belong to $B^s_{p,q}$ and that, in addition, $(u_0,\nabla\theta_0)\in L^r$
for some $r\in]1,\infty[$ if $p=\infty.$
Then \eqref{eq:boussinesq} admits a unique local-in-time
solution $(\theta,u,\nabla P)$ in $\cC_w(]-T,T[;B^s_{p,\infty})$ if $q=\infty$
and in  $\cC(]-T,T[;B^s_{p,q})$ if $q<\infty.$ 
Besides, $\nabla\theta$ and $u$ are in $\cC(]-T,T[;L^r)$ if $(u_0,\nabla\theta_0)\in L^r.$
\smallbreak

Furthermore, the solution may be continued beyond\footnote{For expository purpose, we just
consider  positive times} $T$ whenever one of the following three conditions is satisfied:
\begin{enumerate}
\item[i)] $\Int_0^T\|\nabla u\|_{L^\infty}\,dt<\infty$;
\item[ii)]  $\Int_0^T\bigl(\|\omega\|_{L^\infty}+\|\nabla\theta\|_{L^\infty}\bigr)\,dt<\infty$ and $s\!>\!1+N/p$;
\item[iii)] $N=2,$  $\Int_0^T\|\nabla\theta\|_{L^\infty}\,dt<\infty$ and $s\!>\!1+2/p.$ 
\end{enumerate}
\end{theorem}
Before proving this result, a few comments are in order.
\begin{enumerate}
\item[1.] If it is assumed that
$\omega_0\in L^r$ instead of $u_0\in L^r$ then the vorticity 
of the constructed solution is continuous in time with values in $L^r.$
\item[2.]
In the two-dimensional case and in the H\"older spaces framework, the
above statement has been established in \cite{CKN}. 
The critical Besov case (that is $p=1+2/p$, $p\in]1,\infty[$) has been
investigated in \cite{LWZ}. 
\item[3.]
In \cite{ES}, a continuation criterion involving the $L^\infty$ norm of the vorticity only 
has been stated. However, as the first
inequality below (7) therein fails if $m\geq2,$ we do not know whether that criterion is correct. 
\item[4.] 
The first item  has been proved
recently in \cite{LWZ} in the two dimensional case.
\item[5.] Let us finally mention that one may replace $\|\omega\|_{L^\infty}$
with $\|\omega\|_{\dot B^0_{\infty,\infty}\cap L^r}$ in the second criterion.
 \end{enumerate}
 \smallbreak\noindent\textit{Proof of Theorem \ref{th:boussinesq}.}~
The proof of the local well-posedness in the Besov spaces 
framework is  a straightforward adaptation to that 
of the corresponding result for the Euler system in $B^s_{p,q},$ and is thus omitted. 
The reader may refer to \cite{BCD}, Chap. 7 for more details.
\bigbreak
So let us go for the proof of the continuation criteria. 
Let us first assume that $1<p<\infty.$ In this case, the  Marcinkiewicz theorem for Calderon-Zygmund operators  ensures that 
 \begin{equation}\label{eq:BS}
 \|\nabla u\|_{L^p}\leq C\|\omega\|_{L^p}.
 \end{equation}
 Therefore, decomposing $\omega$ into low and high frequencies as follows\footnote{The notation $\Delta_{-1}$
 is defined in the appendix}:
  $$\omega=\Delta_{-1}\omega+(\Id-\Delta_{-1})\omega,$$
 and taking advantage of the remark that follows Proposition \ref{p:properties} in the appendix, we  gather that
\begin{equation}\label{eq:BSbis}
\|\nabla u\|_{B^{s-1}_{p,q}}\leq C\|\omega\|_{B^{s-1}_{p,q}}.
\end{equation}
 
 Now, in dimension $N$ the vorticity equation reads
$$
\d_t\omega+u\cdot\nabla\omega+\cA(\nabla u,\omega)={}^T\!\nabla(\theta e_N)-\nabla(\theta e_N)
\quad\hbox{with }\ \cA(\nabla u,\omega):=\omega\cdot\nabla u+{}^T\!\nabla u\cdot\omega.
$$
Hence 
applying $\Delta_j$ to the vorticity equation yields
$$
\d_t\omega_j+u\cdot\nabla\omega_j=-\Delta_j\cA(\nabla u,\omega)+\Delta_j\bigl({}^T\!\nabla(\theta e_N)-\nabla(\theta e_N)\bigr)+[u,\Delta_j]\cdot\nabla\omega
$$
with $\omega_j:=\Delta_j\omega$ and $\theta_j:=\Delta_j\theta.$
Therefore, because $\div u=0,$
\begin{multline}\label{eq:above}
\|\omega_j(t)\|_{L^p}\leq\|\omega_j(0)\|_{L^p}+\int_0^t\|\nabla\theta_j\|_{L^p}\,d\tau\\
+\int_0^t\|\Delta_j\cA(\nabla u,\omega)\|_{L^p}\,d\tau+\int_0^t\|[u,\dj]\cdot\nabla\omega\|_{L^p}\,d\tau.
\end{multline}
{}Next, let us use  (see the appendix) that
\begin{equation}\label{eq:com}
\bigl\|2^{j(s-1)}\|[u,\Delta_j]\cdot\nabla\omega\|_{L^p}\bigr\|_{\ell^q}
\lesssim \|\nabla u\|_{L^\infty}\|\omega\|_{B^{s-1}_{p,q}}\quad\hbox{whenever }\ s>0.
\end{equation}
If $s>1+N/p$ then standard tame estimates (see e.g. \cite{BCD}, Chap. 2) imply that
$$
\begin{array}{lll} \|\cA(\nabla u,\omega)\|_{B^{s-1}_{p,q}}&\leq&
C\bigl(\|\omega\|_{L^\infty}\|\nabla u\|_{B^{s-1}_{p,q}}+\|\nabla u\|_{L^\infty}\|\omega\|_{B^{s-1}_{p,q}}\bigr),\\[1ex]
&\leq& C\|\nabla u\|_{L^\infty}\|\nabla u\|_{B^{s-1}_{p,q}}.
\end{array}
$$
The last inequality remains true in the limit case $s=1+N/p$ and $q=1,$
a consequence of the algebraic structure of $\cA(\nabla u,\omega)$
(see e.g. Inequality (52) in \cite{DF}).
\smallbreak

Hence, multiplying \eqref{eq:above} by $2^{j(s-1)},$ taking the $\ell^q$ norm with respect
to  $j$ and taking advantage of \eqref{eq:BSbis} yields
\begin{equation}\label{eq:vort1}
\|\omega(t)\|_{B^{s-1}_{p,q}}\leq \|\omega_0\|_{B^{s-1}_{p,q}}
+C\int_0^t\|\nabla\theta\|_{B^{s-1}_{p,q}}\,d\tau
+C\int_0^t\|\nabla u\|_{L^\infty}\|\omega\|_{B^{s-1}_{p,q}}\,d\tau.
\end{equation}

Next, in order to bound the $B^s_{p,q}$ norm of $\theta,$
we use the fact that
$$
\d_t\theta_j+u\cdot\nabla\theta_j=[u,\Delta_j]\cdot\nabla\theta,
$$
whence 
\begin{equation}\label{eq:above1}
\|\theta_j(t)\|_{L^p}\leq\|\theta_j(0)\|_{L^p}
+\int_0^t\|[u,\dj]\cdot\nabla\theta\|_{L^p}\,d\tau.
\end{equation}
Given that, according to (a slight modification of) Lemma 2.100 of \cite{BCD}, we have
\begin{equation}\label{eq:combis}
\bigl\|2^{js}\|[u,\Delta_j]\cdot\nabla\theta\|_{L^p}\bigr\|_{\ell^q}
\lesssim \|\nabla u\|_{L^\infty}\|\theta\|_{B^{s}_{p,q}}+\|\nabla\theta\|_{L^\infty}
\|\omega\|_{B^{s-1}_{p,q}},
\end{equation}
we eventually get
\begin{equation}\label{eq:theta1}
\|\theta(t)\|_{B^{s}_{p,q}}\leq \|\theta_0\|_{B^{s}_{p,q}}
+C\int_0^t\Bigl(\|\nabla u\|_{L^\infty}\|\theta\|_{B^{s}_{p,q}}
+\|\nabla\theta\|_{L^\infty}
\|\omega\|_{B^{s-1}_{p,q}}\Bigr)
\,d\tau.
\end{equation}
Finally, from the equation for $\theta,$ we easily get
\begin{equation}\label{eq:theta2}
\|\nabla\theta(t)\|_{L^\infty}\leq \|\nabla\theta_0\|_{L^\infty}+\int_0^t\|\nabla u\|_{L^\infty}
\|\nabla\theta\|_{L^\infty}\,d\tau.
\end{equation}
So if $\nabla u$ is in $L^1([0,T[;L^\infty)$ then 
$\nabla\theta$ is in $L^\infty([0,T[\times\R^N).$ 
Therefore, summing up Inequalities \eqref{eq:vort1} and \eqref{eq:theta1}
and using Gronwall's lemma, we easily deduce that
$\|\omega\|_{B^{s-1}_{p,r}}$ and $\|\theta\|_{B^s_{p,r}}$ are bounded 
on $[0,T[.$
To complete the proof of the boundedness of 
the solution in $L^\infty([0,T[;B^s_{p,r}),$ we still have to bound
$u$ in $L^\infty([0,T[;L^p).$ For that,  we use the fact that 
\begin{equation}\label{eq:u}
u(t)=u(0)-\int_0^t\cP(u\cdot\nabla u)\,d\tau
\end{equation}
where $\cP$ stands for the Leray projector over divergence-free vector-fields.
As it is continuous over $L^p$  (recall that $1<p<\infty$), we deduce that 
\begin{equation}\label{eq:uLp}
\|u(t)\|_{L^p}\leq \|u_0\|_{L^p}+C\int_0^t\|\nabla u\|_{L^\infty}\|u\|_{L^p}\,d\tau.
\end{equation}

Now, the standard continuation criterion for hyperbolic PDEs ensures that 
the solution $(\theta,u)$ may be continued beyond $T.$
 \medbreak
 Let us now treat the case where $s>1+N/p$ and 
 \begin{equation}\label{eq:blowup}
 \int_0^T\bigl(\|\omega\|_{L^\infty}+\|\nabla\theta\|_{L^\infty}\bigr)\,dt<\infty.
 \end{equation}
 
  We first bound  $\omega$ and $\nabla\theta$  in $L^\infty([0,T[;L^p)$
  by taking advantage of  \eqref{eq:BS} and of the vorticity and temperature equations.
  We get
 \begin{eqnarray}\label{eq:vort3}
 &&\|\omega(t)\|_{L^p}\leq\|\omega_0\|_{L^p}+\int_0^t\|\nabla\theta\|_{L^p}\,d\tau
 +C\int_0^t\|\omega\|_{L^p}\|\omega\|_{L^\infty}\,d\tau,\\ 
&&\label{eq:theta3}
 \|\nabla\theta(t)\|_{L^p}\leq\|\nabla\theta_0\|_{L^p}
 +C\int_0^t\|\nabla\theta\|_{L^\infty}\|\omega\|_{L^p}\,d\tau.
 \end{eqnarray}
 Hence,
 $$
 \|(\omega,\nabla\theta)(t)\|_{L^p}\leq
 \|(\omega_0,\nabla\theta_0)\|_{L^p}+ C\int_0^t(1+\|(\omega,\nabla\theta)\|_{L^\infty})
  \|(\omega,\nabla\theta)\|_{L^p}\,d\tau.
 $$
 So Gronwall's lemma provides us with a bound for
  $\omega$ and $\nabla\theta$  in $L^\infty([0,T[;L^p).$
 \smallbreak
 Next, we use the following classical logarithmic interpolation inequality
 (see e.g. \cite{BCD}):
 \begin{equation}\label{eq:interpo}
 \|\nabla u\|_{L^\infty}\lesssim\|\omega\|_{L^p\cap L^\infty}
 \log\bigl(e+\|\omega\|_{B^{s-1}_{p,q}}\bigr).
 \end{equation}
 Plugging this inequality in \eqref{eq:vort1} and \eqref{eq:theta1}, 
 and summing up, we get
 $$
 \displaylines{
\quad \|\omega(t)\|_{B^{s-1}_{p,q}}+\|\theta(t)\|_{B^s_{p,q}}
 \leq  \|\omega_0\|_{B^{s-1}_{p,q}}+\|\theta_0\|_{B^s_{p,q}}\hfill\cr\hfill
 +C\int_0^t\bigl(1+\|\nabla\theta\|_{L^\infty}+\|\omega\|_{L^p\cap L^\infty}\bigr)
 \bigl( \|\omega\|_{B^{s-1}_{p,q}}+\|\theta\|_{B^s_{p,q}}\bigr)
  \log\bigl(e+\|\omega\|_{B^{s-1}_{p,q}}\bigr)\,d\tau.\quad}
 $$
 So Osgood's lemma implies that $\|\omega\|_{B^{s-1}_{p,r}}$ and $\|\theta\|_{B^s_{p,r}}$ are bounded on $[0,T[.$ Bounding $\|u\|_{L^p}$ may be done by combining Inequalities \eqref{eq:uLp} and \eqref{eq:interpo}. 
  Hence the solution $(\theta,u)$ may be continued beyond $T.$
 
 \smallbreak
 Let us finally assume that $N=2$ and that
 $$
 \int_0^T\|\nabla\theta\|_{L^\infty}\,dt<\infty.
$$
Then Equation \eqref{eq:vorticityN=2} gives
$$
\|\omega(t)\|_{L^\infty}\leq\|\omega_0\|_{L^\infty}+\int_0^t\|\d_1\theta\|_{L^\infty}\,d\tau.
$$
Hence $\omega\in L^\infty([0,T[\times\R^2)$ 
and the previous continuation 
criterion implies that  the solution $(\theta,u)$ may be continued beyond $T.$ 
\bigbreak
Let us end the proof with a few comments concerning the cases $p=1,\infty.$
If $p=\infty$ and the solution also satisfies
$(\nabla\theta,\omega)\in L^\infty([0,T[;L^r)$ for some $r\in]1,\infty[,$
then arguing as for proving \eqref{eq:BSbis} yields 
$$
\|\nabla u\|_{B^{s-1}_{\infty,q}\cap L^r}\leq C\|\omega\|_{B^{s-1}_{\infty,q}\cap L^r}.
$$
{}From   the vorticity and temperature equations, we get
$$
\begin{array}{lll}
\|\omega(t)\|_{L^r}&\leq& \|\omega_0\|_{L^r}+2\Int_0^t\|\nabla u\|_{L^\infty}\|\omega\|_{L^r}\,d\tau
+2\int_0^t\|\nabla \theta\|_{L^r}\,d\tau,\\[1ex]
\|\nabla\theta(t)\|_{L^r}&\leq&\|\nabla\theta_0\|_{L^r}+\Int_0^t\|\nabla u\|_{L^\infty}\|\nabla\theta\|_{L^r}\,d\tau.\end{array}
$$
So one may conclude that \eqref{eq:vort1} and \eqref{eq:theta1} hold true
if replacing the norm in $B^{s-1}_{\infty,q}$ by the norm in $B^{s-1}_{\infty,q}\cap L^r.$
In order to bound $\|u\|_{B^s_{\infty,q}},$ one may write  that (using Bernstein's inequality
to get the second line),
$$
\begin{array}{lll}
\|u(t)\|_{B^s_{\infty,q}}&\lesssim&\|\Delta_{-1}u(t)\|_{L^\infty}+\|\omega(t)\|_{B^{s-1}_{\infty,q}},\\[1.5ex]
&\lesssim&\|\Delta_{-1}u_0\|_{L^\infty}+\|\Delta_{-1}(u(t)-u_0)\|_{L^r}+\|\omega(t)\|_{B^{s-1}_{\infty,q}}.
\end{array}
$$
Now, according to \eqref{eq:u}, we have
$$
\|u(t)-u_0\|_{L^r}\leq C\int_0^t\|u\|_{L^r}\|\nabla u\|_{L^\infty}\,d\tau.
$$
{}From this, it is easy to complete the proof. 
\smallbreak
Finally, if $p=1$ then embedding ensures that $\nabla\theta$ and $u$ are
 in $L^\infty([0,T[;L^r)$ for some finite~$r,$ so that one may conclude as in the case $p=\infty.$
 \qed


\subsection{Lower bounds for  the lifespan of the solutions to \eqref{eq:boussinesq}}

Let  $(\theta_0,u_0)$  satisfy the assumptions of Theorem \ref{th:boussinesq}. 
Then is is clear 
 that $(\theta^\eps,u^\eps,\nabla \Pi^\eps)$ satisfies \eqref{eq:boussinesq}
on $[T^-/\eps,T^+/\eps]$  with initial data
$$
\theta_0^\eps=\eps^2\theta_0\quad\hbox{and}\quad
u_0^\eps=\eps u_0
$$
if and only if the triplet $(\theta,u,\nabla\Pi)$ defined by
$$
\theta^\eps(t,x):=\eps^2\theta(\eps t,x),\quad
u^\eps(t,x):=\eps u(\eps t,x)\ \hbox{ and }\
\Pi^\eps(t,x):=\eps^2\Pi(\eps t,x)
$$
satisfies \eqref{eq:boussinesq} on $[T^-,T^+]$ with data $(\theta_0,u_0).$ 
\smallbreak
{}From this,  we gather  that
for initial temperature and velocity of size $\eps^2$ and $\eps,$ respectively, 
the lifespan is (at least) of order  $\eps^{-1}.$
\medbreak
The above result is, obviously, independent of the dimension. 
At the same time, in the case $\theta_0\equiv0$
(corresponding to the  incompressible Euler equation)
global existence holds true in dimension $2.$
In the case $\theta_0\not\equiv0,$ the question of
global existence has remained unsolved, even in the two-dimensional case. 
We here want to study whether, nevertheless,
dimension $2$ is somehow ``better''. 
To answer this question, we shall take advantage of 
the fact that the vorticity equation in dimension $2$
has no stretching term: it reduces to
\begin{equation}\label{eq:vorticityN=2}
\d_t\omega+u\cdot\nabla\omega=\d_1\theta.
\end{equation}
Hence, taking advantage of the special a priori estimates
for the transport equation  in Besov spaces \emph{with null regularity index} 
(as discovered by M. Vishik in \cite{V}
and by T. Hmidi and S. Keraani in \cite{HK}), one may write 
\begin{equation}\label{eq:besov0}
\|\omega(t)\|_{B^0_{\infty,1}}
\leq \biggl( \|\omega_0\|_{B^0_{\infty,1}}
+\int_0^t\|\d_1\theta\|_{B^0_{\infty,1}}\,d\tau\biggr)
\biggl(1+C\int_0^t\|\nabla u\|_{L^\infty}\,d\tau\biggr).
\end{equation}
This will be the key to our result below.
\begin{theorem}
Assume that $N=2.$
Let  $(\theta_0,u_0)$ be in $B^s_{p,q}$ with $(s,p,q)$ 
satisfying  \eqref{eq:conditionC}. If $p\in\{1,+\infty\},$  suppose in addition 
that $(\nabla\theta_0,\omega_0)\in L^r$ for some
$1<r<\infty.$ There exists a constant $C$ depending only on $r$
and such that (setting $p=r$ if $p\in(1,+\infty)$), the lifespan $T^*$ of \eqref{eq:boussinesq} satisfies 
$$
T^*\geq \frac{1}{C\|\omega_0\|_{B^0_{\infty,1}\cap L^r}}
\log\biggl(1+\frac12\log\biggl(1+\frac{C\|\omega_0\|_{B^0_{\infty,1}\cap L^r}^2}{\|\nabla\theta_0\|_{B^0_{\infty,1}\cap L^r}}\biggr)\biggr)\cdotp
$$
\end{theorem} 
\begin{proof}
Let us first notice that, according to the continuation criteria
derived in Theorem \ref{th:boussinesq}, 
it suffices to show that if the solution is defined on $[0,T[\times\R^n$ with
$T\leq T^*$ and  $T^*$ as above, then 
$\omega$ and $\nabla\theta$  are bounded in 
$L^\infty(0,T;B^0_{\infty,1}\cap L^r).$ 
\smallbreak
Now,  estimates for 
the transport  equation in Besov spaces
(see e.g. \cite{BCD}, Chap. 3) yield
\begin{equation}\label{eq:besov1}
\|\nabla\theta(t)\|_{B^0_{\infty,1}}\leq \|\nabla\theta_0\|_{B^0_{\infty,1}}
e^{C\int_0^t\|\nabla u\|_{B^0_{\infty,1}}\,d\tau}.
\end{equation}

Of course, standard $L^r$ estimates for the transport equation imply that
\begin{equation}\label{eq:Lr}
\|\omega(t)\|_{L^r}\leq\|\omega_0\|_{L^r}+\int_0^t\|\d_1\theta\|_{L^r}
\quad\hbox{and}\quad
\|\nabla\theta(t)\|_{L^r}\leq\|\nabla\theta_0\|_{L^r}e^{\int_0^t\|\nabla u\|_{L^\infty}\,d\tau}.
\end{equation}

Let us finally  notice that putting together embedding, Inequality \eqref{eq:BS}
and the remark that follows Proposition \ref{p:properties}, we have
$$
\|\nabla u\|_{L^\infty}\lesssim\|\nabla u\|_{B^0_{\infty,1}}\lesssim
\|\omega\|_{B^0_{\infty,1}\cap L^r}.
$$

Therefore, denoting
$$
\Omega(t):=\|\omega(t)\|_{B^0_{\infty,1}\cap L^r}\quad\hbox{and}\quad
\Theta(t):=\|\nabla\theta(t)\|_{B^0_{\infty,1}\cap L^r}
$$
and taking advantage of \eqref{eq:besov0}, \eqref{eq:besov1}, \eqref{eq:Lr},
 we conclude that 
$$\begin{array}{lll}
\Theta(t)&\leq& \Theta_0e^{C\int_0^t\Omega\,d\tau},\\[1ex]
\Omega(t)&\leq&\biggl(\Omega_0+\Int_0^t\Theta\,d\tau\biggr)
\biggl(1+C\int_0^t\Omega\,d\tau\biggr).
\end{array}
$$
Now, plugging the inequality for $\Theta(t)$ in the inequality
for $\Omega(t),$ we get
\begin{equation}\label{eq:Omega}
\Omega(t)\leq \biggl(\Omega_0+t\Theta_0e^{C\int_0^t\Omega\,d\tau}\biggr)
\biggl(1+C\int_0^t\Omega\,d\tau\biggr).
\end{equation}
Let us assume for a while that 
\begin{equation}\label{eq:smallt}
T\Theta_0e^{C\int_0^T\Omega\,d\tau}\leq \Omega_0.
\end{equation}
Then \eqref{eq:Omega} and Gronwall's lemma imply that
\begin{equation}
\Omega(t)\leq 2\Omega_0e^{2Ct\Omega_0}\quad\hbox{for all }
t\in[0,T].
\end{equation}
Therefore, for Condition \eqref{eq:smallt} to be satisfied, it suffices that 
$$
\Theta_0 T
\exp\biggl(e^{2CT\Omega_0}-1\biggr)\leq \Omega_0
$$
that is to say
\begin{equation}\label{eq:X}
X\exp(e^X-1)\leq Y\quad
\hbox{with } \ X:=2CT\Omega_0\ \ \hbox{ and }\ Y=\frac{2C\Omega_0^2}{\Theta_0}\cdotp
\end{equation}
Let us notice that
$$
X\leq e^X-1\leq \exp(e^X-1)-1\quad\hbox{for any }\ X\in\R^+.
$$
Hence Inequality \eqref{eq:X} is satisfied provided that
$$
\exp\bigl(2(e^X-1)\bigr)\leq 1+Y.
$$
So we easily gather from a bootstrap argument that the lifespan
$T^*$ satisfies 
$$
T^*\geq \frac{1}{2C\Omega_0}\log\biggl(1+\frac12\log\biggl(1+\frac{2C\Omega_0^2}{\Theta_0}\biggr)\biggr),
$$
which is exactly the desired inequality. 
\end{proof}
\begin{remark} In the case where the solution is $C^{1,r}$  for some $r\in(0,1)$
(an assumption which is not satisfied in the critical regularity case), 
 one may first write estimates for $\|\omega\|_{L^\infty}$ and $\|\omega\|_{C^r},$
and next use the classical logarithmic inequality for bounding $\|\nabla u\|_{L^\infty}$
in terms of $\|\omega\|_{L^\infty}$ and $\|\omega\|_{C^r}.$
This does not improve the lower
bound for the lifespan, though.
\end{remark}


\section{The axisymmetric incompressible Euler equations}

We now consider the \emph{incompressible Euler equations}
\eqref{eq:euler}. As recalled in the introduction, 
Euler equations are  globally well-posed in dimension $2.$
In dimension $d\geq3,$  the global well-posedness issue
  has remained unsolved 
unless some property of symmetry is satisfied : it is known 
that axisymmetric or helicoidal \emph{without swirl} data generate global solutions  (see e.g. \cite{D} and the
references therein for more details). 

In the general case, an easy scaling argument
similar to that of the Boussinesq system 
yields that for data of size $\eps,$ the lifespan is at least of order $\eps^{-1}.$
\smallbreak
Here we want to focus on the axisymmetric solutions to  Euler equations \emph{with swirl},
that is on solutions $u$ to \eqref{eq:euler} such that, in cylindrical coordinates,
\begin{equation}\label{eq:axi}
u(r,z)=u^r(r,z)e_r+u^\theta(r,z)e_\theta+u^z(r,z)e_z.
\end{equation}
Recall that the corresponding vorticity reads
$\omega(r,z)=\omega^r(r,z)e_r+\omega^\theta(r,z)e_\theta
+\omega^z(r,z)e_z$
with 
$$
\omega^r(r,z)=-\d_zu^\theta e_r,\quad
\omega^\theta(r,z)=\d_zu^r-\d_ru^z,\quad
\omega^z(r,z)=\frac1r\d_r(ru^\theta).
$$
With this notation,  axisymmetric solutions satisfy
(see e.g. \cite{BM})
\begin{equation}\label{eq:EulerS}
\left\{\begin{array}{l}
\tilde D_tu^r+\d_r\Pi=r^{-1}(u^\theta)^2,\\[1ex]
\tilde D_tu^\theta=-r^{-1}u^ru^\theta,\\[1ex]
\tilde D_tu^z+\d_z\Pi=0,\\[1ex]
\d_r(ru^r)+\d_z(ru^z)=0,\end{array}\right.
\quad\qquad\hbox{with }\  \tilde D_t:=\d_t+u^r\d_r+u^z\d_z.
\end{equation}
As pointed out  in \cite{ES}, there is a striking similarity
between the two-dimensional Boussinesq system \eqref{eq:boussinesq} satisfied by $(\theta,\omega)$
in the previous section, and the equations satisfied by $(u^\theta,\omega^\theta)$ here.
Indeed, 
$$\left\{
\begin{array}{l}
\tilde D_t(ru^\theta)=0,\\[1ex]
\tilde D_t\omega^\theta-\frac1ru^r\omega^\theta-\frac1r\d_z(u^\theta)^2=0
\end{array}\right.
$$
whence, denoting $\Gamma:=(ru^\theta)^2$ and $\zeta:=r^{-1}\omega^\theta,$ we have
\begin{equation}\label{eq:EulerSS}
\tilde D_t\Gamma=0\quad\hbox{and}\quad
\tilde D_t\zeta=\frac1{r^4}\d_z\Gamma.
\end{equation}

 Therefore, up to the singular coefficient $1/r^4,$ 
 the functions $\Gamma=\Gamma(r,z)$  and $\zeta=\zeta(r,z)$ play the same role as the temperature
and the vorticity, respectively,  in the 2D Boussinesq system. 
Keeping in mind that data such that $u_0^\theta\equiv0$ generate global solutions,
it is natural to study whether having $r^{-1}\omega_0^\theta=\cO(1)$
and $ru_0^\theta=\cO(\eps)$ gives rise to a 
family of solutions with lifespan going to infinity when $\eps$ goes to $0.$
\smallbreak
For  technical reasons however, due to the singularity near the 
axis, we shall consider the axisymmetric Euler equations 
in a smooth bounded axisymmetric domain $\Omega$ of $\R^3$ such that, 
for some given $0<r_0<R_0,$
\begin{equation}\label{eq:domain}
\Omega\subset\bigl\{(x,y,z)\in\R^3: r_0<\sqrt{x^2+y^2}<R_0\bigr\}\cdotp
\end{equation}
Let us first give  a local well-posedness result for the Euler equation in a domain:
\begin{theorem}\label{theorem:euler0} Let $(s,p,q)$ satisfy Condition \eqref{eq:conditionC}.
Let $u_0$ be in $B^s_{p,q}(\Omega)$ with $\div u_0=0$ and  $u_0$
 tangent to the boundary of $\Omega.$ 
 Then System \eqref{eq:euler} with slip boundary conditions 
 has a unique local solution $u$ in $\cC_w(]-T,T[;B^s_{p,q}(\Omega))$
 (or in  $\cC(]-T,T[;B^s_{p,q}(\Omega))$ if $q<\infty$). 
 
 If in addition $u_0$ satisfies \eqref{eq:axi} then $u$ satisfies \eqref{eq:EulerS}.
\end{theorem}
\begin{proof}
This statement has been essentially proved  by A. Dutrifoy in \cite{Du} except in the critical case $s=1+3/p$ and $r=1.$
However, the critical case may be handled by the same method\footnote{Proving a continuation
criterion involving the vorticity was the main purpose of Dutrifoy's paper, and this requires
that $s>1+3/p.$ This is probably the
reason why the statement in the critical case is not given therein.} as it relies on a priori estimates
for transport equations which are also true in this case.

The last part of the statement is a classical consequence of the uniqueness and of the symmetry of
the data $u_0.$
\end{proof}

One can now state the main result of this part. 
\begin{theorem}\label{theorem:euler}
Let $u_0$ be an axisymmetric divergence-free vector-field
in $B^s_{p,q}(\Omega)$ with $(s,p,q)$ satisfying 
\eqref{eq:conditionC} and $\Omega$ a bounded domain satisfying \eqref{eq:domain}.
Suppose in addition that $u_0|_{\d\Omega}$ is tangent to the boundary of $\Omega.$ 
 Then the lifespan $T^*$ to the solution of \eqref{eq:euler}
satisfies
$$
T^*\geq \frac{1}{C\|\omega_0^\theta\|_{B^0_{\infty,1}}}\log\biggl(1
+\frac12\log\biggl(1
+\frac{C\|\omega_0^\theta\|_{B^0_{\infty,1}}}{\|(u_0^\theta)^2\|_{B^1_{\infty,1}}}\biggr)\biggr)
$$
for some constant $C$ depending  only  on $\Omega.$ 
\end{theorem}
\begin{proof}
It suffices to bound the norm of $u$ in $B^1_{\infty,1}$ 
as it controls high norms (see \cite{Du} and
notice that $B^1_{\infty,1}$ embeds in $C^{0,1}$). 
Let $\tilde u:=u^re_r+u^ze_z.$
Denote by $\tilde\psi$ the solution given by Proposition \ref{p:elliptic}
 to the elliptic equation 
$$\left\{
 \begin{array}{ccc}
-\Delta\tilde\psi=\omega^\theta e_\theta&\hbox{ in }& \Omega,\\[1ex]
\d_n\tilde\psi=0&\hbox{ on }& \d\Omega,\end{array} 
\right.\qquad\qquad\int_\Omega \tilde\psi\,dx=0.
$$
Notice that $$\div\tilde u=0=\div(\nabla\wedge\tilde\psi)\quad\hbox{and that}\quad
\nabla\wedge\tilde u=\omega^\theta e_\theta=\nabla\wedge(\nabla\wedge\tilde\psi).
$$
As, in addition,  both $\tilde u$ and $\nabla\wedge\tilde\psi$ have null circulation on 
the components of $\d\Omega$ (a consequence of the symmetry properties
of those two functions and of the domain), they coincide.
Hence, Proposition \ref{p:elliptic} ensures that 
\begin{equation}\label{eq:BS1}
\|\nabla\tilde u\|_{B^0_{\infty,1}}\leq C\|\omega^\theta e_\theta\|_{B^0_{\infty,1}}.
\end{equation}
This inequality will enable us to adapt to the axisymmetric Euler equations
the proof of  lower bounds for the lifespan of solutions.

We proceed as follows.  According to the work by A. Dutrifoy (see
in particular Prop. 6 and Cor. 5 in \cite{Du}) for the transport equation
in a smooth bounded domain, estimates in Besov spaces $B^s_{p,q}(\Omega)$ 
are the same as in the whole space case.
From this, one may deduce by following the method of \cite{HK} 
that in  the particular case $s=0$, the estimates improve (as in \eqref{eq:besov0}). So we get,
 bearing   \eqref{eq:EulerSS} in mind:
$$
\|\zeta(t)\|_{B^0_{\infty,1}}\leq\biggl(\|\zeta_0\|_{B^0_{\infty,1}}
+\Int_0^t\|r^{-4}\d_z\Gamma\|_{B^0_{\infty,1}}\,d\tau\biggr)
\biggl(1+C\int_0^t\|\nabla \tilde u\|_{L^\infty}\,d\tau\biggr).
$$
General Dutrifoy's estimates for the transport equation 
also imply that
$$
\|\Gamma(t)\|_{B^1_{\infty,1}}\leq\|\Gamma_0\|_{B^1_{\infty,1}}
\exp\biggl(C\Int_0^t\|\nabla \tilde u\|_{B^0_{\infty,1}}\,d\tau\biggr).
$$
Now, the important observation is that $1/r^4$ is in $C^{0,1}(\Omega)$ (for $r\geq r_0$ in 
$\Omega$). Hence 
$$
\|r^{-4}\d_z\Gamma\|_{B^0_{\infty,1}}\leq C\|\d_z\Gamma\|_{B^0_{\infty,1}},
$$
whence 
$$
\|\zeta(t)\|_{B^0_{\infty,1}}\leq \biggl(\|\zeta_0\|_{B^0_{\infty,1}}
+C\Int_0^t\|\d_z\Gamma\|_{B^0_{\infty,1}}\,d\tau\biggr)
\biggl(1+C\int_0^t\|\nabla \tilde u\|_{L^\infty}\,d\tau\biggr).
$$
Finally, according to  \eqref{eq:BS1} and classical  embedding properties, we have
$$
\|\nabla \tilde u\|_{L^\infty}\lesssim \|\nabla\tilde u\|_{B^0_{\infty,1}}
\lesssim \|\omega^\theta e_\theta\|_{B^0_{\infty,1}}.
$$
As $\omega^\theta e_\theta=\zeta\, r e_\theta$ and, under our assumption on 
$\Omega,$ $r e_\theta$ is in $C^{0,1},$ one may thus conclude that 
$$
\|\nabla \tilde u\|_{L^\infty}\lesssim  \|\zeta\|_{B^0_{\infty,1}}.
$$ 
{}From this point, one may proceed exactly as for the Boussinesq system;
we deduce the following lower bound for the lifespan of the solution:
$$
T^*\geq \frac{1}{C\|\zeta_0\|_{B^0_{\infty,1}}}\log\biggl(1
+\frac12\log\biggl(1
+\frac{C\|\zeta_0\|_{B^0_{\infty,1}}}{\|\Gamma_0^2\|_{B^1_{\infty,1}}}\biggr)\biggr)\cdotp
$$
Of course, owing to the shape of $\Omega,$ up to an irrelevant constant, 
one may replace $\zeta_0$
with $\omega^\theta_0$ and $ru_0^\theta$ with $u_0^\theta,$ respectively.
\end{proof}

\begin{remark}
We believe Theorem \ref{theorem:euler} to be true in the case where $\Omega$ satisfying \eqref{eq:domain}
is unbounded. However, we refrained from giving  the statement as we
did not find in the literature the counterpart of Theorem \ref{theorem:euler0}
and of Proposition \ref{p:elliptic}. 

Let us emphasize however that unbounded domains have been considered in \cite{BF} (H\"older spaces), 
and \cite{J,K} (weighted Sobolev spaces).   By following Dutrifoy's approach, we do not see any obstruction 
to get similar results in the Besov space framework. This is only a matter
of having suitable extension operators available for the domain considered.

We also believe that Proposition \ref{p:elliptic} may be extended to unbounded domains provided
we prescribe some condition at infinity: the following inequality  
\begin{equation}\label{eq:BS2}
\|\nabla\tilde u\|_{B^0_{\infty,1}\cap L^r}\leq C\|\omega^\theta e_\theta\|_{B^0_{\infty,1}\cap L^r}.
\end{equation}
for any $r\in]1,+\infty[$ seems to be reasonable. However, as proving such inequalities 
is not the point of this paper, we restricted ourselves to bounded domains. 
\end{remark}

\subsection*{Acknowledgment} The author is indebted to O. Glass and  F. Sueur for 
pointing out references \cite{J,K,Triebel}.


\appendix
\section{}
\setcounter{equation}{0}

In this  appendix, we   recall the definition and a few properties of nonhomogeneous Besov spaces $B^s_{p,q},$
then  prove a commutator estimate.
\medbreak

Let us first introduce  a  dyadic partition of unity 
with respect to the Fourier variable (the so-called Littlewood-Paley decomposition): we 
fix a smooth radial function $\chi$ supported in (say) the ball $B(0,4/3),$ 
equals to $1$ in a neighborhood of $B(0,3/4)$
and such that $r\mapsto\chi(r\,e_r)$ is nonincreasing
over $\R_+,$ and set
$\varphi(\xi)=\chi(\xi/2)-\chi(\xi).$
\smallbreak
The {\it dyadic blocks} $(\Delta_j)_{j\in\Z}$
 are defined by
$$
\dj:=0\ \hbox{ if }\ j\leq-2,\quad\Delta_{-1}:=\chi(D)\quad\hbox{and}\quad
\Delta_j:=\varphi(2^{-j}D)\ \text{ if }\  j\geq0.
$$
It may be easily checked that the identity  $u=\sum_{j}\dj u$ holds true in the sense
of tempered distributions.

One can now define the  Besov space $B^s_{p,q}$ as the set of tempered distributions $u$ 
so that $\|u\|_{B^s_{p,q}}$ is finite, where
$$
\|u\|_{B^s_{p,q}}:=\bigg(\sum_{j} 2^{qjs}
\|\Delta_j  u\|^q_{L^p}\bigg)^{\frac{1}{q}}\ \text{ if }\ q<\infty
\quad\text{and}\quad
\|u\|_{B^s_{p,\infty}}:=\sup_{j}\left( 2^{js}
\|\Delta_j  u\|_{L^p}\right).
$$

 Roughly speaking, the elements of $B^s_{p,q}$ have ``$s$ derivatives in $L^p$''. 
 For instance,  the Besov space $B^s_{2,2}$ coincides
  with the nonhomogeneous Sobolev space $H^s$  (for any $s\in\R$), and  $B^s_{\infty,\infty}$ coincides
  with the H\"older space $C^s,$ if $s\in\R_+\setminus\N.$
\smallbreak
In this paper, we use freely the following properties for Besov spaces (see e.g. \cite{BCD}, Chap. 2):
\begin{proposition}\label{p:properties}
 Let$(s,p,q)\in\R\times[1,+\infty]^2.$
\begin{itemize}
\item The Besov space  $B^s_{p,q}$ is (continuously) embedded
in the set $C^{0,1}$ of Lipschitz  bounded functions if and only if
 Condition \eqref{eq:conditionC} is satisfied.
 \item  The gradient operator
maps $B^s_{p,q}$ in $B^{s-1}_{p,q}.$
\item  More generally, if  $F:\R^N\rightarrow\R$
is a  smooth homogeneous function of degree $m$ away from a neighborhood of the origin
 then for all $(p,q)\in[1,\infty]^2$ and $s\in\R,$  Operator $F(D)$ maps $B^s_{p,q}$
in $B^{s-m}_{p,q}.$ 
\end{itemize}
\end{proposition}
\begin{remark}
{}From the last property, 
given that  the  Biot-Savart operator $\cB:\omega\mapsto\nabla u$
is an homogeneous smooth multiplier of degree $0,$  we deduce that 
$(\Id-\Delta_{-1})\cB$ is a  self-map on $B^s_{p,q}$ \emph{for any $s\in\R$ and $1\leq p,q\leq\infty.$} This implies Inequality \eqref{eq:BSbis}. 
\end{remark}

The definition of Besov spaces may be extended by restriction 
to general domains $\Omega$ of $\R^N$:
\begin{definition} Let $\Omega$ be a domain of $\R^N,$ and $(s,p,q)\in\R\times[1,+\infty]^2.$
We denote by $B^s_{p,q}(\Omega)$ the set of distributions $u$ over $\Omega$
which are the restriction (in the sense of distributions) to some $\tilde u$ in $B^s_{p,q}(\R^N).$
The space $B^s_{p,q}(\Omega)$ is endowed with the norm
$$
\|u\|_{B^s_{p,q}(\Omega)}:=\inf \|\tilde u\|_{B^s_{p,q}(\R^N)}
$$
where the infimum is taken over the set of $\tilde u$ in $B^s_{p,q}(\R^N)$ such that $u$ coincides
with the restriction of $\tilde u$ to $\Omega.$
\end{definition}

\smallbreak
The following result will be needed in the proof of Inequality \eqref{eq:BS1}.
\begin{proposition}\label{p:elliptic}
Let $\Omega$ be a smooth bounded domain of $\R^N$
and  $\omega$ be in $B^0_{\infty,1}(\Omega).$
If in addition the mean value of $\omega$ on $\Omega$ is zero then the Neumann equation
$$\left\{
\begin{array}{lll}
-\Delta\psi=\omega&\hbox{ in }& \Omega,\\[1ex]
\d_n\psi=0&\hbox{ on }& \d\Omega,\end{array} 
\right.\qquad\qquad\int_\Omega \psi\,dx=0,
$$
has a unique solution $\psi$ in $B^2_{\infty,1}(\Omega)$ and we have
$$
\|\nabla^2\psi\|_{B^0_{\infty,1}(\Omega)}\leq C\|\omega\|_{B^0_{\infty,1}(\Omega)}.
$$
\end{proposition}
\begin{proof} 
In \cite{Triebel}, Th. 4.4, it has been proved that, for any $s>-1,$
 if $\omega\in C^s(\Omega):=B^s_{\infty,\infty}(\Omega)$  (with $0$ mean value) then the above system 
 has a unique solution $\psi$ in $C^{s+2}(\Omega)$ satisfying 
 $$
 \|\nabla^2\psi\|_{C^s(\Omega)}\leq C\|\omega\|_{C^s(\Omega)}.
 $$
 So denoting $T:\omega\mapsto\nabla^2\psi,$ the result follows by interpolation : it is only a matter of using
the fact that  $B^0_{\infty,1}(\Omega)=(C^{-1/2}(\Omega),C^{1/2}(\Omega))_{1/2,1}.$
\end{proof}
\smallbreak
Let us now turn to the proof of Inequality \eqref{eq:com}. 
Let $\tilde u:=u-\Delta_{-1}u.$ 
We decompose the commutator as follows:
\begin{equation}\label{eq:dec}
[u,\Delta_j]\cdot\nabla\omega=\sum_{i=1}^6 R_j^i
\end{equation}
with, using the summation convention over repeated indices,
$$
\begin{array}{lll}
R_j^1:=[T_{\tilde u^k},\Delta_j]\d_k\omega,&&
R_j^2:=T_{\d_k\dj\omega}\tilde u^k,\\[1ex]
R_j^3:=-\dj T_{\d_k\omega}\tilde u^k,&&
R_j^4:=\d_kR(\tilde u^k,\dj \omega),\\[1ex]
R_j^5:=-\d_k\dj R(\tilde u^k,\omega),&&
R_j^6:=[\Delta_{-1}\tilde u^k,\dj]\d_k\omega.
\end{array}
$$
Above,  $T$ and  $R$ stand for the paraproduct and remainder
operators, respectively,  which are defined as follows (after J.-M. Bony
in \cite{Bony}):
$$
T_fg:=\sum_jS_{j-1}f\dj g\  \hbox{ and }\ 
R(f,g):=\sum_j\sum_{|j'-j|\leq1}\dj f\,\Delta_{j'}g\quad \hbox{ with }\
S_j:=\sum_{j'<j}\dj.
$$
Decomposition \eqref{eq:dec} is obtained after noticing that 
\begin{equation}\label{eq:bony}
fg=T_fg+T_gf+R(f,g).
\end{equation}

Let us now go to the proof of 
Inequality \eqref{eq:com}. 
In all that follows, $(c_j)_{j\geq-1}$ stands for a sequence
such that $\|(c_j)\|_{\ell^q}=1.$
\smallbreak
{}From \cite{BCD},   Lemma 2.99, we get, for $i\in\{1,6\},$
$$
\|R_j^i\|_{L^p}\lesssim c_j2^{-j(s-1)}\|\nabla u\|_{L^\infty}\|\omega\|_{B^{s-1}_{p,q}}.
$$
As regards $R_j^2,$ we write that
$$
R_j^2=\sum_{j'\geq j-1} S_{j'-1}\d_k\dj\omega\,\Delta_{j'}\tilde u^k.
$$
Note that $\cF(\dj\omega)$ is supported in an annulus of size $2^j$.
Hence Bernstein's Inequality ensures that 
$$
\|R_j^2\|_{L^p}\lesssim 2^j\sum_{j'\geq j-1} \|\omega\|_{L^\infty}\|\Delta_{j'}\tilde u\|_{L^p},
$$
whence
$$
\|R_j^2\|_{L^p}\lesssim 2^{-j(s-1)} \|\omega\|_{L^\infty}
\sum_{j'\geq j-1}2^{(j-j')s}\:2^{j's}\|\Delta_{j'}\tilde u\|_{L^p},
$$
so that we get if $s>0,$ 
$$
\|R_j^2\|_{L^p}\lesssim c_j2^{-j(s-1)}\|\omega\|_{L^\infty}\|\tilde u\|_{B^{s}_{p,q}}.
$$

As for $R_j^3,$ standard continuity results for the paraproduct operator 
(see e.g. \cite{BCD}, Chap. 2)
imply that $$
\|R_j^3\|_{L^p}\lesssim c_j2^{-j(s-1)}\|\omega\|_{L^\infty}\|\tilde u\|_{B^{s}_{p,q}}.
$$
For $R_j^4,$ one may write that
$$
R_j^4=\d_k\sum_{|j'-j|\leq 2}
\Delta_{j'}\tilde u^k\,\dj(\Delta_{j'\!-\!1}\!+\!\Delta_{j'}\!+\!\Delta_{j'\!+\!1})\omega.
$$
Hence, in view of Bernstein inequality, 
$$
\|R_j^4\|_{L^p}\lesssim2^j\sum_{|j'-j|\leq 2}
\|\omega\|_{L^\infty} \|\Delta_{j'}\tilde u\|_{L^p}.
$$
So  we get 
$$
\|R_j^4\|_{L^p}\lesssim c_j2^{-j(s-1)}\|\omega\|_{L^\infty}\|\tilde u\|_{B^{s}_{p,q}}.
$$
Next, standard continuity results for the remainder operator yield if $s>0,$
$$
\|\d_kR(\tilde u^k,\omega)\|_{B^{s-1}_{p,q}}\lesssim
 \|\omega\|_{L^\infty}\|\tilde u\|_{B^{s}_{p,q}}.
$$
Hence 
$$
\|R_j^5\|_{L^p}\lesssim c_j2^{-j(s-1)}\|\omega\|_{L^\infty}\|\tilde u\|_{B^{s}_{p,q}}.
$$
Finally, let us notice that the operator $\omega\mapsto (\Id-\Delta_{-1})u$ satisfies the hypothesis of the
last item of Proposition \ref{p:properties} with $m=-1,$  hence
$$
\|\tilde u\|_{B^{s}_{p,q}}\lesssim\|\omega\|_{B^{s-1}_{p,q}}.
$$
So putting all the above inequalities together  
completes the proof of \eqref{eq:com}.


\bibliographystyle{amsplain}

\end{document}